\documentclass{article}

\usepackage{hyperref}
\usepackage[affil-it]{authblk}

\usepackage{float}
\usepackage[margin=1in]{geometry} 
\usepackage[normalem]{ulem}
\usepackage{amsmath,amsthm,amssymb, graphicx, multicol, array}
\usepackage{xcolor}
\usepackage{enumitem}
\usepackage{tikz}
\usepackage{graphicx}
\newtheorem{thm}{Theorem}[section]

\newtheorem{lemma}{Lemma}[section]

\title{Fractals Generated by Modifying Aperiodic Substitution Tilings}
\author{}

\begin{document}

\maketitle

\begin{abstract}
This study proposes a method for producing an infinite number of fractals using aperiodic substitution tilings, exemplified by the Ammann Chair tiling. Higher order substitutions of aperiodic tilings are utilized in relation to the Sierpinski carpet concept. The similarity dimensions of the fractals generated by the Ammann Chair tiling are calculated and shown to be dense. A fractal image generator was implemented in the Java programming language and is freely available for public use at \url{https://github.com/KahHengLee/Ammann-Chair-Fractal.git}.

\end{abstract}

\section{Introduction}

A tiling is a plane covering composed of compact sets of tiles that are almost disjoint, in the sense that no two of them have common interior points. On the other hand, a fractal is a complex geometric shape with fine structure at arbitrarily small scales, often exhibiting a degree of self-similarity. Fractals are explained and illustrated in the book Nonlinear Dynamics and Chaos, specifically in chapter 11 \cite{textbook}. This paper aims to connect the concepts of aperiodic tilings and fractals and proposes a method for building fractals from tilings. The research on tilings contributes to various fields, such as art, geometry, and even crystallography.

Our study specifically focuses on the Ammann Chair tiling \cite{Ammann}, which exhibits self-similarity properties that make similarity dimension calculations much simpler. Unlike the well-known aperiodic Penrose tiling, the Ammann Chair tiling does not lack self-similarity properties. The boundary of the Penrose tiling changes with each iteration, making it difficult to determine its similarity dimension.

In this paper, we also reference the well-known plane fractal, the Sierpinski Carpet, which was constructed in 1916 by Waclaw Sierpinski \cite{sierpinski1916courbe}. The fractal's main idea is to remove a certain area (the middle square) in each iteration, eventually forming a complex geometric shape with fine holes at arbitrarily small scales after many iterations.

\begin{figure} [H]
\centering
\includegraphics[scale=0.5]{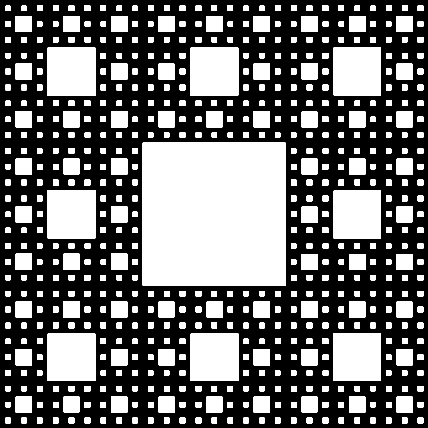}
\text{ }
\includegraphics[scale=0.2]{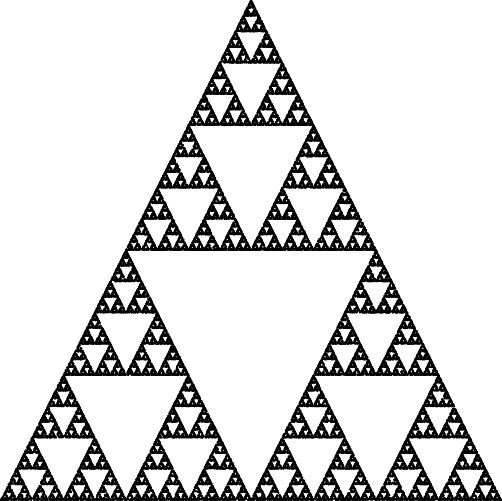}
\caption{Sierpinski Carpet and Sierpinski Triangle.}
\end{figure}

\begin{figure} [H]
    \centering
    \includegraphics[scale=0.09]{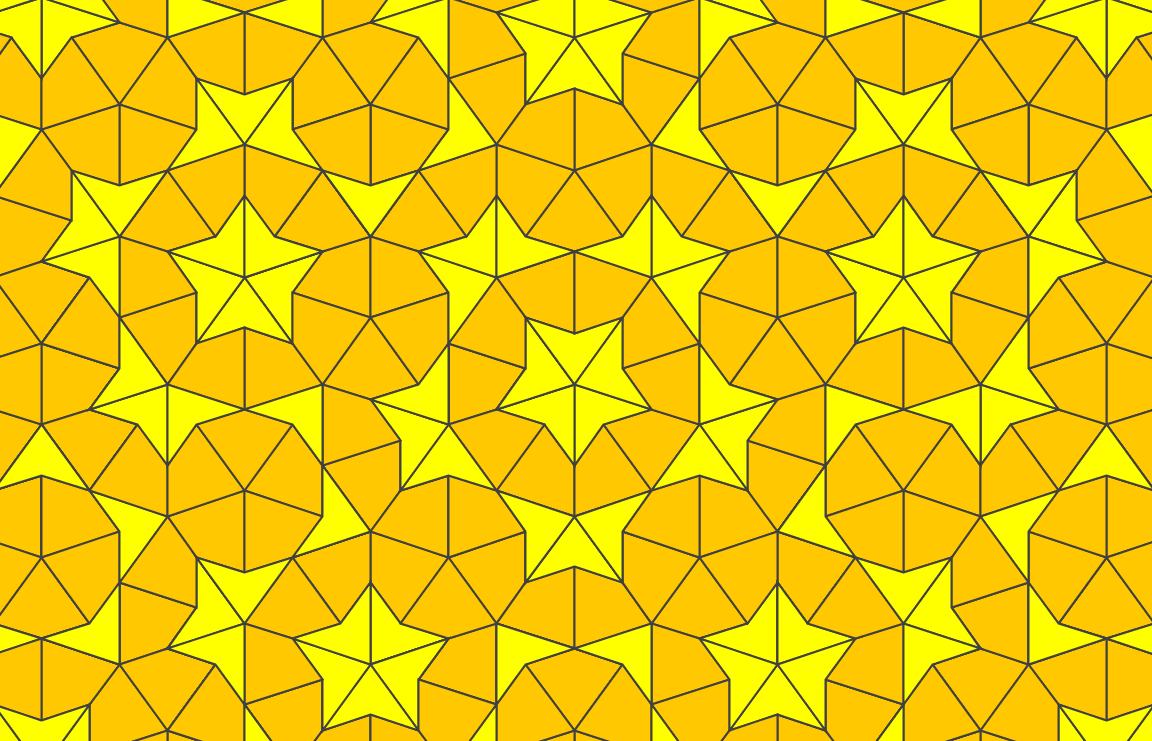}
    \text{ }
    \includegraphics[scale=0.2]{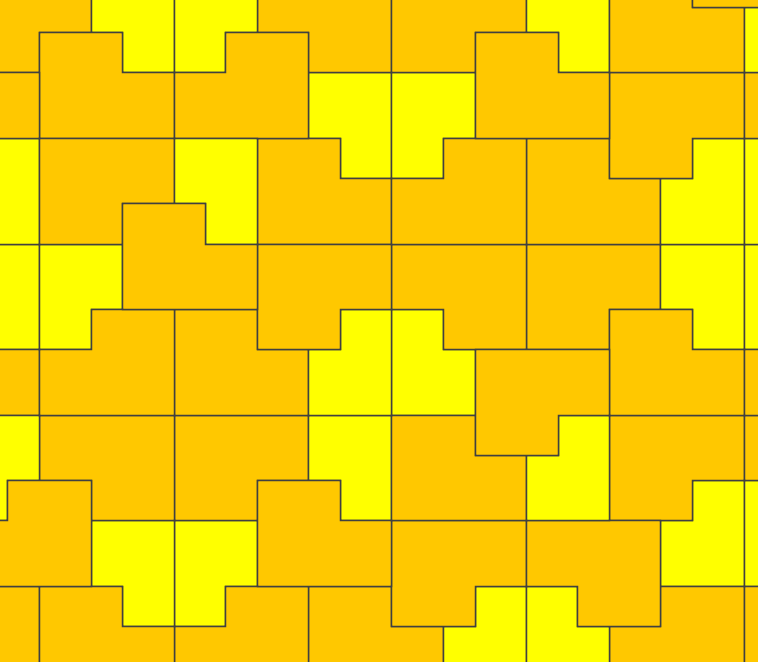}
    \caption{Penrose kite and dart tiling (left), and Ammann Chair tiling (right).}
\end{figure}

\section{Similar Research}
\label{sec:similar research}
Several studies have explored the relationship between fractals and aperiodic tiling, with a particular focus on the Penrose tiling.

Liu's research in physics \cite{Zhengyou_1995} investigates the scaling and scaling-related dynamical properties of the Penrose tiling. Analytically, the Penrose tiling has a fractal dimension of $d_f=2$, which is equivalent to its Euclidean dimension (in $\mathbb{R}^2$). Additionally, the paper provides numerical evidence that the physics equality holds on the Penrose lattice. This equality describes the relationship between spectral dimension, fractal dimension, and diffusive dimension. Liu's research not only provides insights into the physical properties of the Penrose lattice but also analytical calculations that can be used to determine the fractal dimension of the Ammann Chair tiling.

In the publication Fractal Dual Substitution Tilings \cite{Fractaldual}, the authors demonstrate a method for creating an infinite number of new fractal tilings. They achieve this by creating new tilings through graph-iterating function systems, and the resulting tiles have fractal boundaries. Like our paper, this article shows how to create an infinite number of new fractals, but with a focus on modifying the tile boundary rather than plane fractal.

Another article discusses a modification of the Penrose tiling to achieve self-similar properties. This involves changing the shape of the kite and dart (Penrose tiling tiles) from quadrilateral to spiky fractals. The modification includes self-similar properties as well as perfect matching rules \cite{fractalpenrose}. This research provides an option for extending our paper, namely how to replace the Ammann Chair tiling with the Penrose tiling while preserving self-similarity.

\section{Ammann Chair Tiling}
\label{sec:Ammann Chair Tiling}
We can generate the Ammann Chair tiling using an algorithm that replaces existing tiles (old) with new tiles. This method of generating tiling is known as the substitution method.

\begin{figure}[H]
\centering
\includegraphics[scale=0.35]{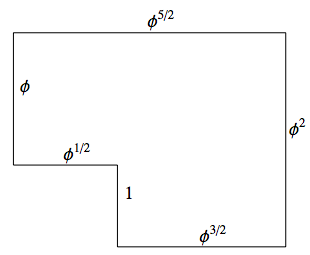}
\caption{Tiles of the Ammann Chair tiling.}
\end{figure}

\subsection{The Typical Substitution Method}
\label{sec:sub}
The Ammann Chair tiling is typically constructed using two sizes of tiles, which we refer to as the ``big tile" (orange) and the ``small tile" (yellow). The size of the big tile is larger than the size of the small tile by a factor of $\phi^{\frac{1}{2}}$, which is determined during the substitution process.

The area of the Ammann Chair tiling increases over generations with the inflation factor of $\phi^\frac{1}{2}$, eventually forming a tiling of the plane \cite{Ammann}. However, for technical reasons, we consider a tiling of the initial tile. By scaling down the size of the tiles over generations, the area of the tiling stays the same. These two approaches are closely related, except for the difference in size. Both of them have the same combinatorial and geometric structures.

\begin{figure}[H]
\centering
\label{fig:substition_rule}
\includegraphics[scale=0.15]{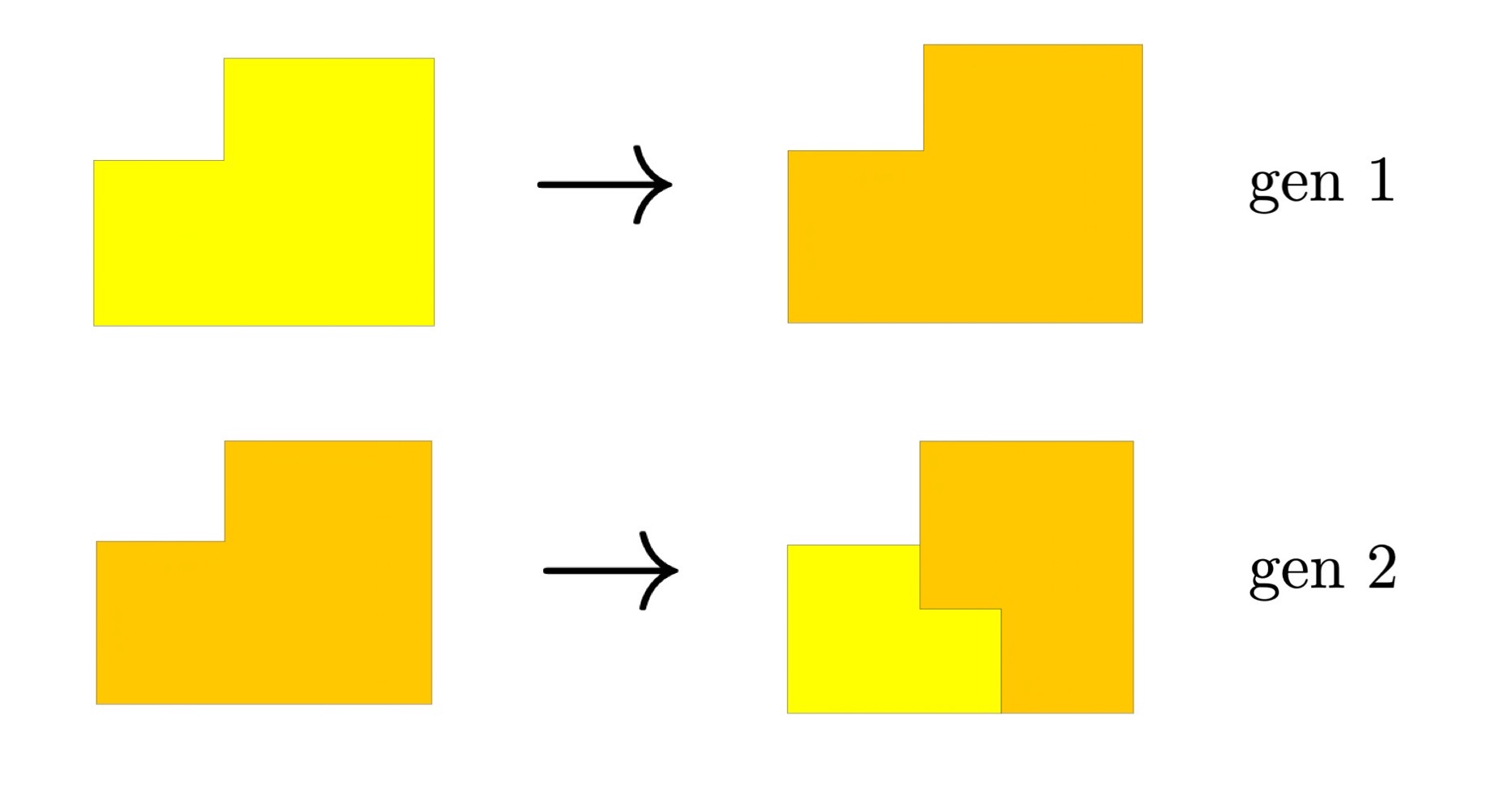}
\caption{The typical substitution algorithm of the Ammann Chair tiling.}
\end{figure}

In subsequent sections, we will refer to the $n$-th generation of the Ammann Chair tiling on the typical substitution by ``Gen $n$".

\begin{figure}[H]
\centering
\includegraphics[scale=0.12]{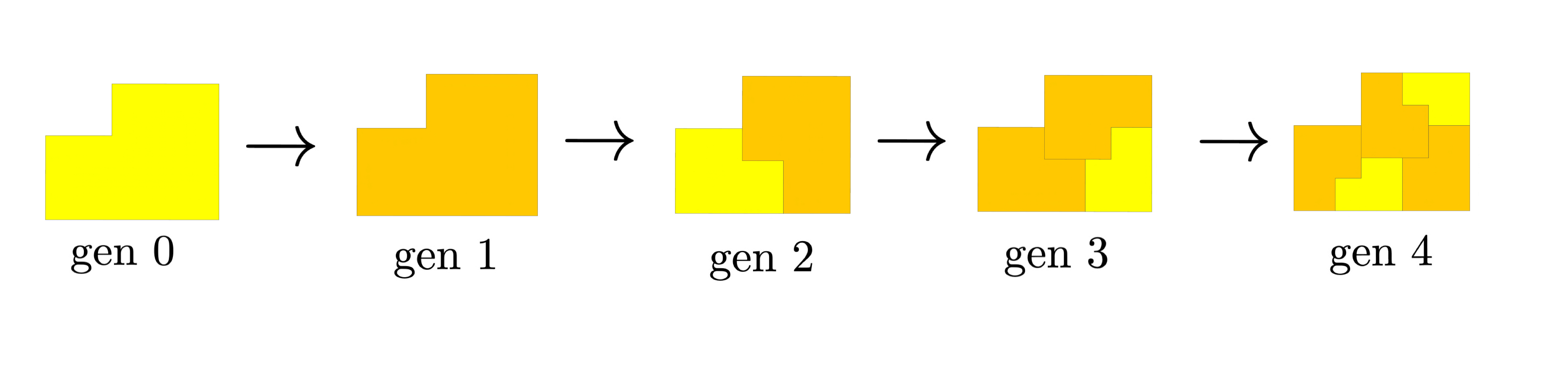}
\caption{Several generations of the Ammann Chair tiling.}
\end{figure}

\subsection{The n-Substitution Method}
Instead of taking one step per iteration, we take two steps, defining this higher-order substitution as 2-substitution. We can easily observe that carrying out 2-steps at once (2-substitution) produces the same tiling as carrying out 1-step at once (1-substitution) using the properties of taking an infinite number of iterations. This concept can be extended to any integer $n$ and is referred to as $n$-substitution.

The $n$-substitution algorithm is easily expressed by considering the transformed (new) tile to be the $n$-th generation of its original (old) tile. Furthermore, by performing $n$ steps at a time, we expect deflation properties to affect the tiles $n$ times in each iteration. As a result, the inflation factor for $n$-substitution is $\phi ^{\frac{n}{2}}$. While generations in n-substitution growth $n$ times faster than typical $1$-substitution. For example, the $k$th generation of $n$-substitution is identical to Gen $n*k$ in $1$-substitution.

\begin{figure}[H]
\centering
\includegraphics[scale=0.15]{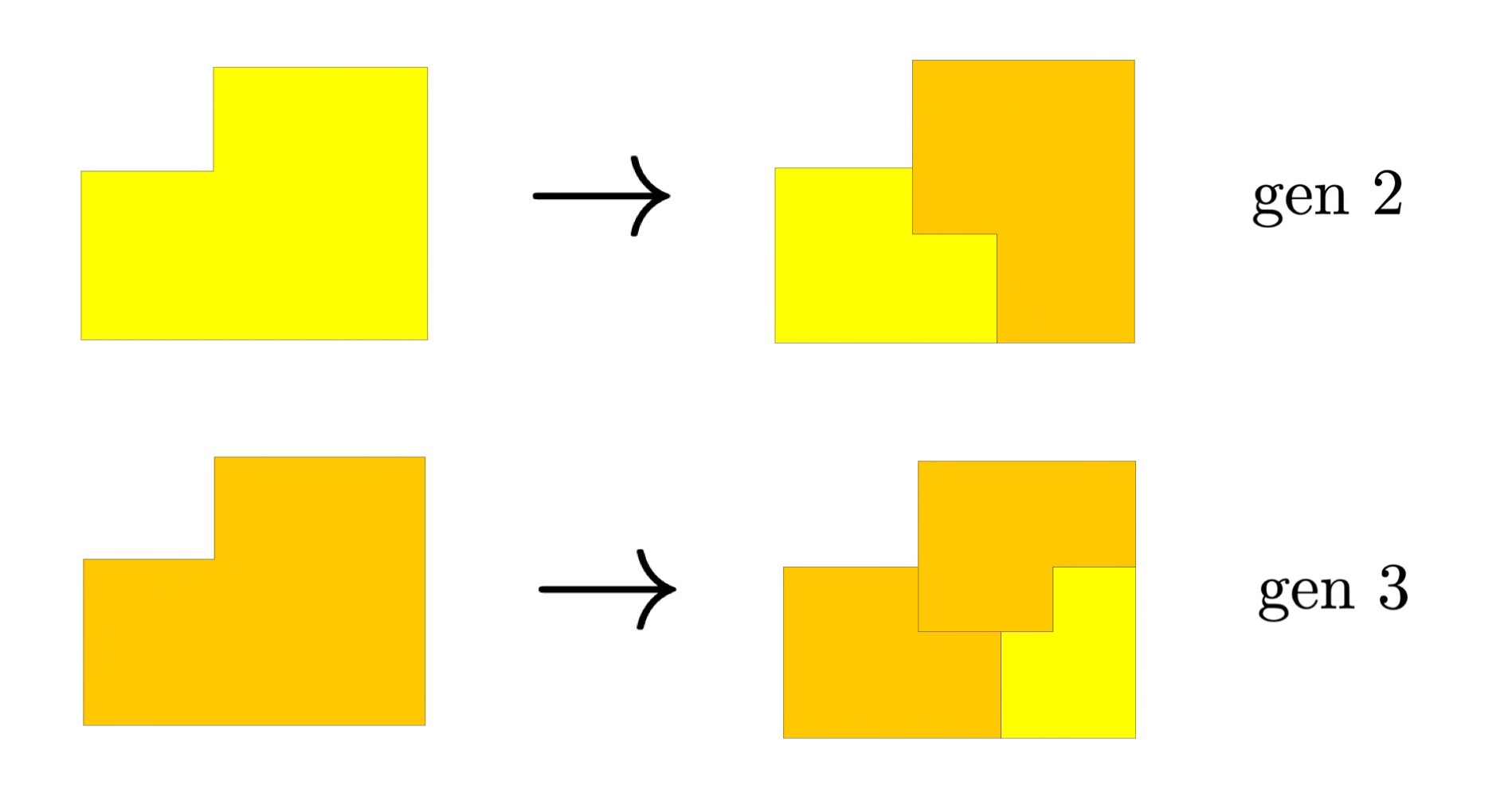}
\caption{The 2-substitution algorithm of the Ammann Chair tiling.}
\end{figure}

\begin{figure}[H]
\centering
\includegraphics[scale=0.15]{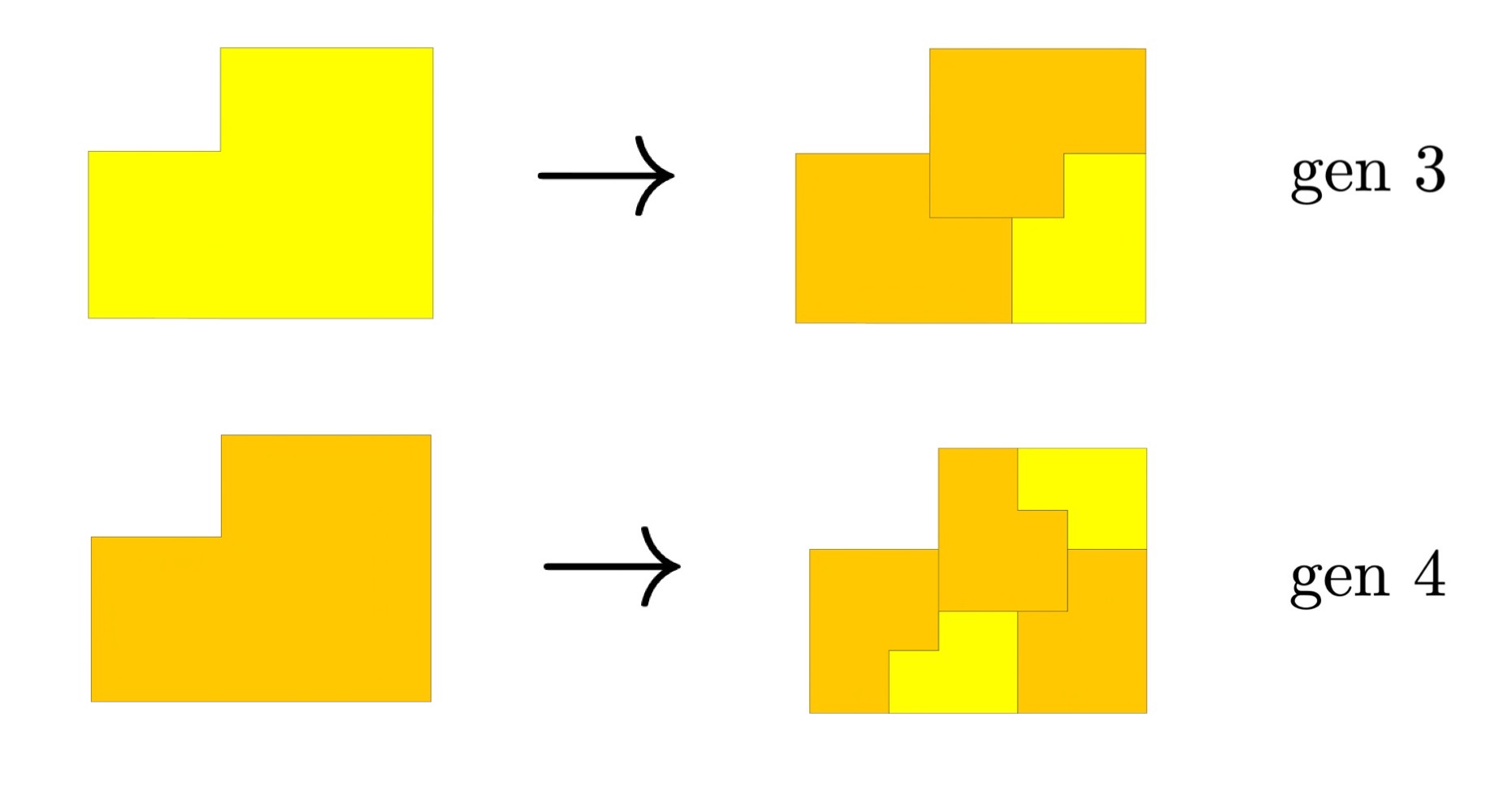}
\caption{The 3-substitution algorithm of the Ammann Chair tiling.}
\end{figure}

\section{Generating a Fractal from Ammann Chair Tiling}
\label{sec:similarity dimension of n-sub}
Up to this point, we have established that the Ammann Chair tiling is an aperiodic tiling in two dimensions and does not exhibit many characteristics of a fractal. In this section, we will draw inspiration from the work of Wacaw Sierpiski and use the Ammann Chair tiling to create a new fractal.

The Ammann Chair tiling becomes increasingly fine with each generation, while the overall shape remains constant. This property is similar to that of a well-known fractal tiling, the Sierpinski Carpet. For more information on how the Sierpinski Carpet works, see the book Nonlinear Dynamics and Chaos \cite{textbook}.

\begin{figure}[H]
\centering
\includegraphics[scale=0.2]{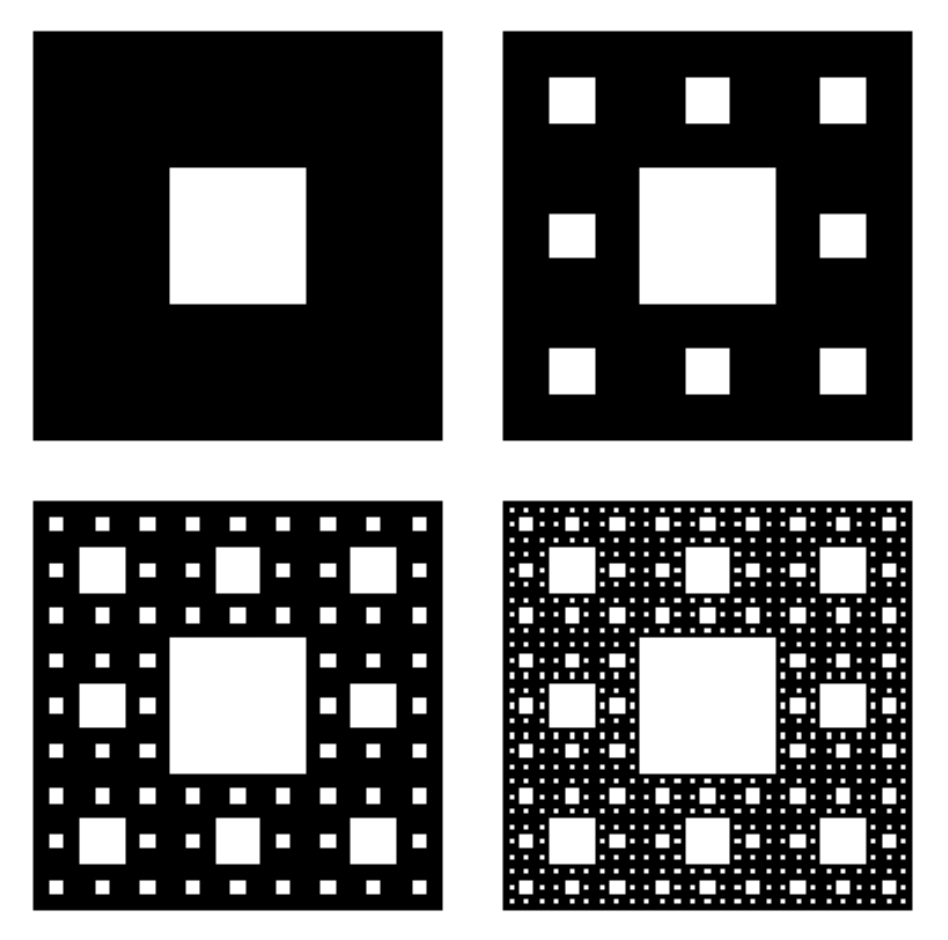}
\caption{Sierpinski Carpet from generation $1$ to generation $4$.}
\end{figure}

\subsection{Counting Tiles in Gen n}
\label{sec:tile counting}
In $n$-substitution, each tile is replaced by multiple tiles per iteration. It turns out that this substitution is related to Gen $n$. After one generation, Gen $n$ replaces a small tile, while Gen $n+1$ replaces a large tile. Therefore, the number of tiles in Gen $n$ (or Gen $n+1$) plays a role in determining the similarity dimension.

Recall that in the typical substitution method ($1$-substitution), Amman Chair tiling does 2 types of replacements. A small tile is replaced by a big tile, whereas a big tile is replaced by a small and a big tile. This can be rewitten as follows:

\begin{align}
1S &\to 1B\\
1B &\to 1S + 1B
\end{align}

Let $S_n$ be the number of small tiles in Gen $n$ and $B_n$ be the number of large tiles in Gen $n$. By simple counting and induction, we can write $S_n$ and $B_n$ as follows:

\begin{align}
S_n &= F_{n-1} S_0+F_{n}B_0\\
B_n &= F_{n} S_0+F_{n+1}B_0,
\end{align}

where $F_n$ refers to the $n$-th term of the Fibonacci sequence.

Consider the typical case of the Ammann Chair tiling, which begins with a single small tile, so $S_0=1$ and $B_0=0$. Then, we have:

\begin{align}
S_n &= F_{n-1}\\
B_n &= F_{n}
\end{align}

\subsection{Removing Tiles}
\label{sec:take away tiles}
The Sierpinski Carpet has a fine square grid with all the squares painted black, but the middle squares are removed from the grid. As the generations progress, the grid becomes finer, and more small middle squares are removed. Using the same logic, we can remove some tiles from the Ammann Chair tiling in each iteration to generate a fractal.

Consider a specific case of $5$-substitution. We remove a large tile from Gen $5$. Then, we use the substitution algorithm shown in Figure \ref{fig:sub5_takeaway} to form the fractal shown in Figure \ref{fig:sub5_fractal}. We name the substitution algorithm as $(5,0,1)$-substitution, and the fractal as $(5,0,1)$-Ammann Chair Fractal. We remove tiles from an earlier generation (Gen 5 instead of Gen 6) to ensure the region being taken away is defined as a complete tile(s) in both generations. At \url{https://github.com/KahHengLee/Ammann-Chair-Fractal.git}, an $(n,a,b)$-Ammann Chair Fractal image generator written in Java is available.

\begin{figure} [H]
    \centering
    \includegraphics[scale=0.35]{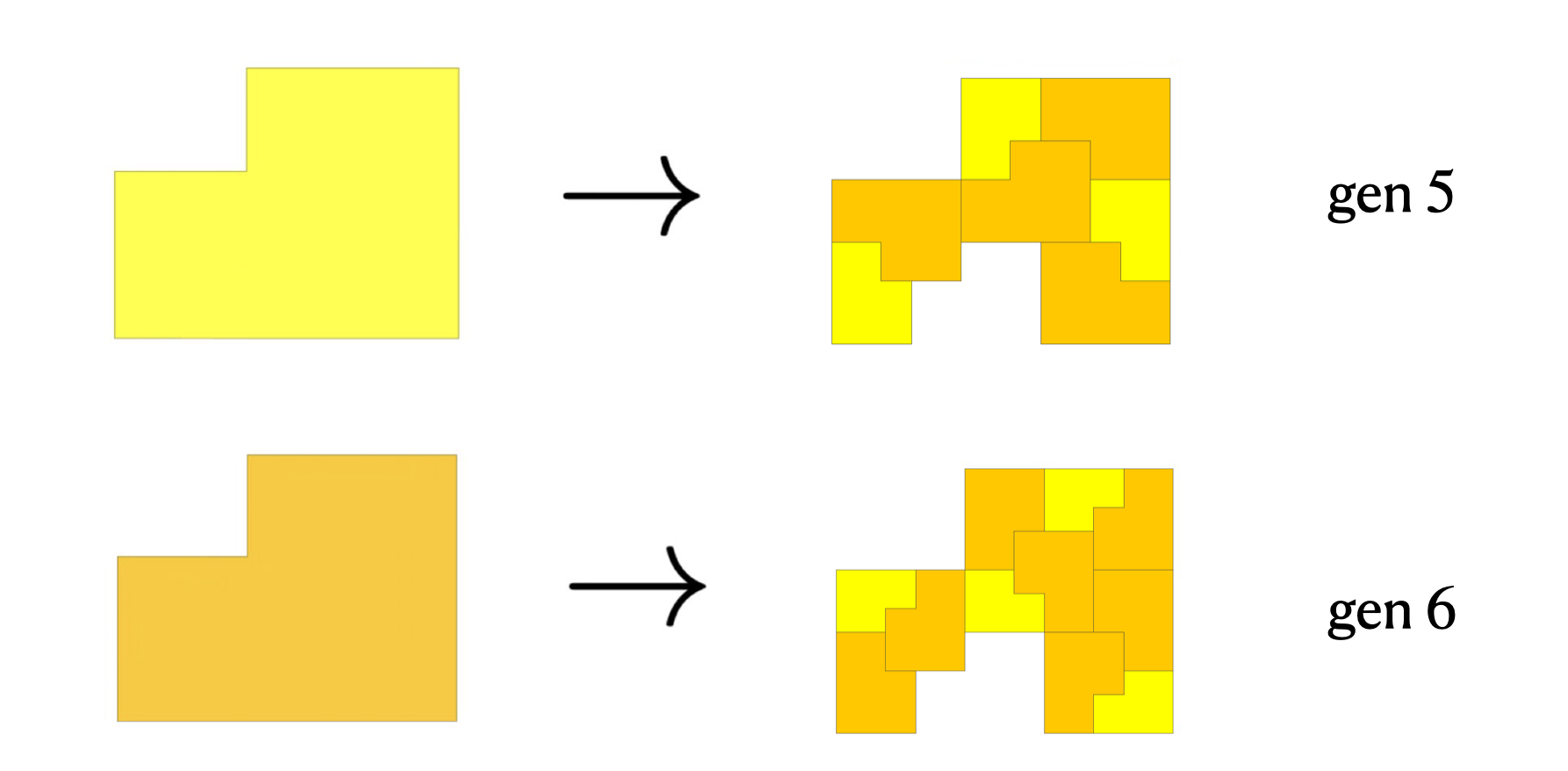}
    \caption{$(5,0,1)$-substitution algorithm.}
    \label{fig:sub5_takeaway}
\end{figure}
\begin{figure} [H]
    \centering
    \includegraphics[scale=0.5]{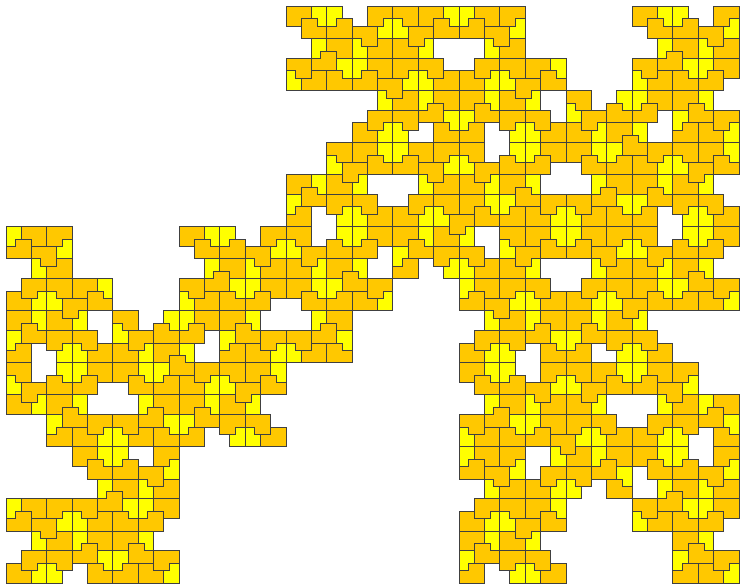}
    \caption{$(5,0,1)$-Ammann Chair fractal.}
    \label{fig:sub5_fractal}
\end{figure}

\section{Similarity Dimension}
Recall that the inflation factor of the $n$-substitution tiling is $r = \phi^{\frac{n}{2}}$, meaning that a tile scaled by $r$ will be the same size as its previous generation. Using the self-similar properties and tile counting from Section \ref{sec:similarity dimension of n-sub}, we can now determine the similarity dimension of the Ammann Chair Fractal. Further reading on the similarity dimension can be found in the book Nonlinear Dynamics and Chaos \cite{textbook}.

For simple illustration, we consider the specific case that appeared in Section \ref{sec:take away tiles}: the $5$-substitution with one big tile removed. 

\vskip 0.3em \noindent \textbf{Calculation.} 
Refer to Figure \ref{fig: 5-sub scaling}, we relates red highlighted tile on the left as a small tile in Gen $5$ and the whole fractal relates to Gen $5$. Scale the small tile region (red highlighted) by $\phi^\frac{5}{2}$, we have $F_5-1=5-1=4$ copies in big tile region and $F_4=3$ in small tile region. We can easily determine its similarity dimension with this property.

\begin{figure}[H]
\label{fig: 5-sub scaling}
\centering
\includegraphics[scale=0.4]{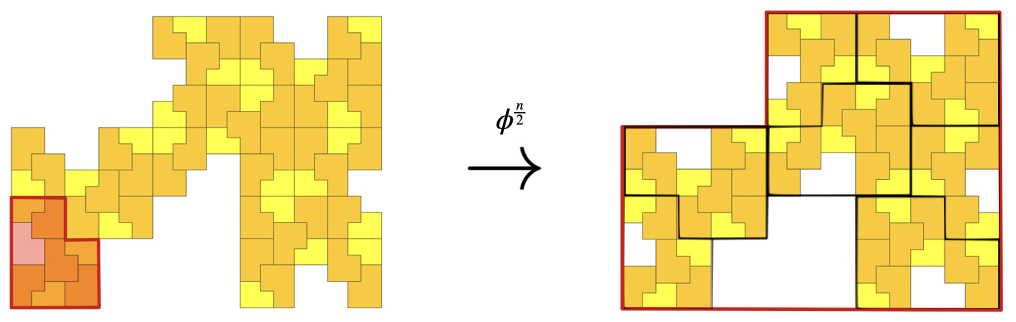}
\caption{Illustration of scaling a small tile in the $4$-substitution.}
\end{figure}

In this context, we are considering the fractal in the region covered by the tiles. When we have a scaled tile, its area has the same dimension as that of the whole fractal. Therefore, when we consider the big tile as a scaled version of the small tile, we need to take into consideration its dimension $d$, multiply by $\phi^\frac{d}{2}$ instead of $\phi$.

Let $A$ be the size of the fractal in the small tile region. Then we have the following equation:
\begin{align}
{(\text{scaling factor } r)}^{(\text{ dimension } d)} &= \text{\#copies}\\
{(\phi^{\frac{n}{2}})}^d &= \text{\#small tiles}+(\phi^\frac{1}{2})^d*\text{\#big tiles}\\
{(\phi^{\frac{n}{2}})}^d &= F_4+\phi^\frac{d}{2}(F_4-1)\\
{(\phi^\frac{5}{2})}^d &= 3+4\phi^\frac{d}{2}\\
\end{align}
Let $x=\phi^\frac{d}{2}$, and solve for $x$ and $d$:
\begin{align}
x^5 - 4x -3 &= 0\\
\phi^\frac{d}{2} = x &\approx 1.5600\\
d &\approx 2 \log_\phi 1.5600\\
d &\approx 1.848
\end{align}
As expected, this fractal has a similarity dimension less than 2, as we removed some regions from a 2D object. It should be noted that this calculation works even if we took a different big tile from the $(5,0,1)$-substitution, since similarity dimension is defined in terms size, which is not affected by a rigid transformation.

\subsection{Generalization of the Similarity Dimension}

The concept presented in the previous section is based on the definition of the similarity dimension and is applicable for a higher-order substitution. This allows us to determine the similarity dimension of all fractals generated with an $n$-substitution.

For any integer $n$, consider an $n$-substitution with the absence of $a$ small tiles and $b$ big tiles. We have

\begin{align}
{(\phi^{\frac{n}{2}})}^d &=\text{number of small tiles}+\sqrt{\phi}^d*\text{number of big tiles}\\
\label{eqn:ori eqn}
{(\phi^{\frac{d}{2}})}^n &=(F_{n-1}-a)+\phi^{\frac{d}{2}}(F_{n}-b)\\
0 &= {(\phi^{\frac{d}{2}})}^n - (F_{n-1}-a)- \phi^{\frac{d}{2}}(F_{n}-b)\\
\text{Let } x=\phi^{\frac{d}{2}}, \\
0 &= x^n - (F_{n}-b)x- (F_{n-1}-a)
\end{align}

We can now determine the similarity dimension by solving for $x$. However, it is an n-degree polynomial, which means it has $n$ solutions. In this situation, we choose a positive real root because $x$ is defined as a positive $\phi$ raised to the power of a non-negative $d$. To ensure that the choice of roots is well-defined, we can prove that the polynomial has only one real positive solution for all possible values of $n$, $a$, and $b$.

\begin{lemma}
\label{lmm:unique solution on p(x)}
For any $n\in \mathbf{N}$, $a \in \mathbf{N} \cap [0,F_{n-1}]$, and $b \in \mathbf{N} \cap [0,F_{n}]$, the polynomial
\begin{equation}
    p(x)=x^n - (F_{n}-b)x- (F_{n-1}-a)
\end{equation}
has only one positive real root.
\end{lemma}

\begin{proof}
Taking the double derivative of $p(x)$, we have:

\begin{align}
p(x) &= x^n - (F_{n}-b)x - (F_{n-1} -a)\\
p'(x) &= nx^{n-1} - (F_{n}-b)\\
p''(x) &= n(n-1) x^{n-2}
\end{align}

For $x>0$, $p''(x)>0$, which shows that $p(x)$ is concave upward for $x>0$. By the range of $a$, we know $p(0)\leq 0$. This proves that $p(x)$ has only one positive real root.
\end{proof}

For more convenient discussion in the later sections, we define the fractal generated by an n-substitution method with Ammann Chair tiles as an $(n,a,b)$-Ammann Chair Fractal, where $a$ represents the number of small tiles taken and $b$ represents the number of big tiles taken. With Lemma \ref{lmm:unique solution on p(x)}, we can determine the similarity dimension of an $(n,a,b)$-A the similarity dimension of $(n,a,b)$-Ammann Chair Fractal.

\begin{thm}
\label{thm:fractal dimension}
For any $n\in \mathbf{N}$, $a \in \mathbf{N} \cap [0,F_{n-1}]$, $b \in \mathbf{N} \cap [0,F_{n}]$, except the case of all tiles were taken ($a=F_{n-1} \wedge b=F_{n}$). An $(n,a,b)$-Ammann Chair Fractal has similarity dimension 
\begin{equation}
\label{eqn:dimension in thm}
    d(n,a,b) = 2* \log _\phi x,
\end{equation}
where $x$ is the unique positive real root of $p(x)=x^n-(F_{n}-b)x-(F_{n-1}-a)$.
\end{thm}

\section{Distribution of Fractal Dimensions}

In the previous section, we obtained various fractals by selecting different values for $n$, $a$, and $b$ in the $(n,a,b)$-Ammann Chair Fractal. It is natural to wonder if it is possible to generate a fractal with any desired dimension. If this is the case, then the fractals would be more flexible and applicable to a wider range of mathematical and even real-world problems.

To answer this question, we first need to observe the injective relationship between the $(n,a,b)$-Ammann Chair Fractal and its similarity dimension. In set theory, we say that the cardinality of the domain is larger than the cardinality of the codomain.

As we know, the number of tiles in Gen $n$ is just $F_n+F_{n-1}$. This implies that the set of all possible fractals generated with n-substitution is a finite set. Considering that $n$ can be any natural number, we have a countably infinite number of $n, a, b$-Ammann Chair Fractals. This result corresponds to the fact that the cardinality of both the domain and the codomain is countable. Therefore, it is not an interval and does not contain any interval.

\subsection{Distribution of Fractal Dimension for Fixed n}
\label{sec:dimension for fixed n}

Since a fractal generated from 2D tiles, we might intuitively think that its similarity dimension will fall between 1 and 2. For the upper bound, taking away tiles in the $n$-substitution only decreases the dimension, which implies that the similarity dimension has an upper bound equal to 2. Surprisingly, for the lower bound, we find that it is not valued at 1, but instead is lower bounded by 0.

From Theorem \ref{thm:fractal dimension}, $d(n,a,b)$ is expressed as a function, where for any natural number $n$, $a\leq F_{n-1}$, $b\leq F_{n}$, and except the case of all tiles were taken ($a=F_{n-1} \wedge b= F_{n}$). Considering a fixed value of $n$, we can plot the similarity dimension against the number of tiles taken. The graph shows the trend of the similarity dimension as more tiles are taken away. For the sake of comparison, the x-axis is scaled as the ratio of the number of taken tiles to the total number of tiles. Consider a specific case where $n=10$, and the tiles are taken in order from all available small tiles to big tiles.

\begin{figure} [H]
\centering
\label{graph:fractal dimension of 10-sub}
\includegraphics[scale=0.6]{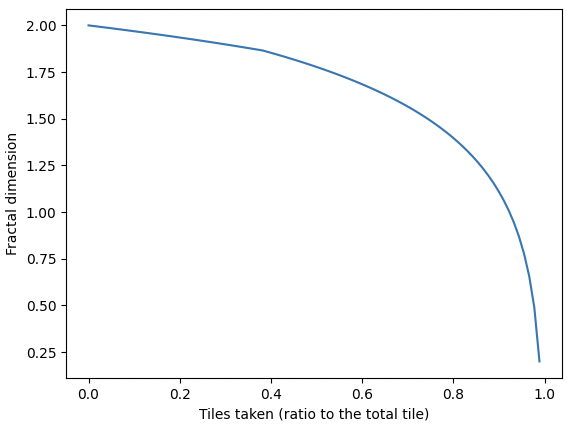}
\caption{Illustration of scaling a small tile in the $4$-substitution.}
\end{figure}

Figure \ref{graph:fractal dimension of 10-sub} shows a significant downward trend, as the similarity dimension strictly decreases as more tiles are taken away. This matches our reasoning, as there are fewer tiles in the $n$-substitution, resulting in a more dramatic decrease in area after each substitution. The remarkable thing about this graph is that the rate of decrease is getting faster and faster along the x-axis. This is equivalent to saying that the function has a negative strictly decreasing derivative.

\subsection{Density of Fractal Dimension}

Consider any $(n,a,b)$-Ammann Chair Fractal. The presence of a single tile has little effect on its similarity dimension since a tile occupies only a small area in Gen $n$, particularly in cases with large $n$. This observation suggests that when $n$ is sufficiently large, the difference in similarity dimensions between adjacent
values of $a$ and $b$ becomes small, even forming a dense set in the interval $[0,2]$. By introducing a lemma, we can determine the density of dimensions of the Ammann Chair Fractal.

\begin{lemma}
\label{lmm:d=d'}
For all $(n,a,b)$-Ammann Chair Fractals with dimension $d$, for all $k\in \mathbb{N}$, there always exists an $(n+k,a',b')$-Ammann Chair Fractal with the same dimension $d$, where
\begin{equation}
\begin{bmatrix}
    a' \\
    b' \\
\end{bmatrix}
=
\begin{bmatrix}
    0 & 1 \\
    1 & 1 \\
\end{bmatrix}^k
\begin{bmatrix}
    a\\
    b\\
\end{bmatrix}
\end{equation}.
\end{lemma}

Lemma \ref{lmm:d=d'} is trivial from Equation \ref{sec:tile counting}, it shows that for any $(n,a,b)$-Ammann Chair Fractal, there always exist fractals with a higher order of (n+k)-substitution that share the same dimension. Notice that the tiles are finer for higher values of $n$, and the absence of a tile has less effect on its dimension.

\begin{thm}
The set of similarity dimensions for $(n,a,b)$-Ammann Chair Fractals is dense in $[0,2]$.
\end{thm}

\begin{proof}
Consider an $(n,a,b)$-Ammann Chair Fractal with dimension $d$. By Equation \ref{eqn:ori eqn}, we have

\begin{align}
\left(\phi^{\frac{n}{2}}\right)^d=\phi^{\frac{d}{2}}\left(F_n-b\right)+\left(F_{n-1}-a\right)
\end{align}

By Lemma \ref{lmm:d=d'}, $\forall k, \exists a', b'$ such that $(n+k,a',b')$-Ammann Chair Fractal has dimension $d$, we have

\begin{align}
\begin{aligned}
\left(\phi^{\frac{a+k}{2}}\right)^d&=\phi^{\frac{d}{2}}\left(F_{n+k}-b'\right)+\left(F_{n+k-1}-a'\right) \\
\left(\phi^{\frac{n}{2}}\right)^d&=\left(\phi^{\frac{d}{2}}\right)^{-k}\left[\phi^{\frac{d}{2}}\left(F_{n+k}-b'\right)+\left(F_{n k-1}-a'\right)\right]
\label{eqn:(1)}
\end{aligned}
\end{align}

Without loss of generality, we modify the $(n+k,a',b')$-Ammann Chair Fractal by adding (or removing) a small (or big) tile. Let $d'$ be its dimension after modification, we can write

\begin{align}
\begin{gathered}
\left(\phi^{\frac{n+k}{2}}\right)^{d'}=\phi^{\frac{d'}{2}}\left(F_{n+k}-b'\right)+\left(F_{n+k-1}-a'\right)+1 \\
\left(\phi^{\frac{n}{2}}\right)^{d'}=\left(\phi^{\frac{d'}{2}}\right)^{-k}\left[\phi^{\frac{d'}{2}}\left(F_{n+k}-b'\right)+\left(F_{n+k-1}-a'\right)\right]+\left(\phi^{\frac{d'}{2}}\right)^{-k}.
\end{gathered}
\end{align}

As $k\rightarrow \infty$, $\left(\phi^{\frac{d'}{2}}\right)^{-k} \rightarrow 0$. By taking limit of $k$ goes to infinity, we have

\begin{align}
\lim _{k \rightarrow \infty}\left(\phi^{\frac{n}{2}}\right)^{d'}=\lim _{k \rightarrow \infty}\left(\phi^{\frac{d'}{2}}\right)^{-k}\left[\phi^{\frac{d'}{2}}\left(F_{n+k}-b'\right)+\left(F_{n+k-1}-a'\right)\right]
\label{eqn:(2)}
\end{align}

By treating the equations as polynomials of degree $n$, Equation \ref{eqn:(2)} converges to Equation \ref{eqn:(1)} in terms of coefficients. With Proposition 5.2.1 from the textbook ``Algebra" by Artin \cite{artin}, we conclude that $d'$ converges to $d$. Therefore, the collection of all similarity dimensions for the $(n,a,b)$-Ammann Chair Fractal is dense in $[0,2]$.
\end{proof}

\section{Future Research}
In this paper, we have analyzed the Ammann Chair tiling as an example of an aperiodic tiling. The new generation tiles were scaled to the same size as their predecessors, and their similarity dimension was calculated. However, this method does not allow for the analysis of non-fixed-boundary tilings, such as the Penrose tiling, which have a spiky boundary that changes over generations.

A promising direction for future research is to generalize the concept and apply it to other types of tiles. The self-similar Penrose tiling presented in the article Fractal Penrose Tilings I: Construction and Matching Rules \cite{fractalpenrose} could be useful in studying the direction proposed in Section \ref{sec:similar research}. It is expected that the similarity dimension of a tiling depends on the similarity dimension of its boundary.

\bibliographystyle{unsrt}
\bibliography{Test} 

\end{document}